\documentclass[11pt]{amsart}

\usepackage{amssymb}
\usepackage{xcolor}
\usepackage{mathabx}
\usepackage{shuffle}
\usepackage[OT2,T1]{fontenc}
\DeclareSymbolFont{cyrletters}{OT2}{wncyr}{m}{n}
\DeclareMathSymbol{\Sha}{\mathalpha}{cyrletters}{"58}
\usepackage{amsmath}
\usepackage[top=1.2in, bottom=1.2in, left=1.4in, right=1.4in]{geometry}
\input xy
\xyoption{all}
\usepackage{amsfonts}
\usepackage{mathrsfs}
\usepackage{tikz-cd}
\usepackage{graphicx}
\usepackage{amscd,amsbsy,amsthm}
\usepackage[all]{xy}
\usepackage[colorlinks,plainpages,urlcolor=magenta]{hyperref}
\usepackage{verbatim}
\usepackage{enumitem}
\usepackage[normalem]{ulem}
\usepackage{tikz}
\usetikzlibrary{arrows,calc}
\tikzset{
  >=stealth',
  help lines/.style={dashed, thick},
  axis/.style={<->},
  important line/.style={thick},
  connection/.style={thick, dotted},
}
\usepackage{stackengine,scalerel}

\makeatletter

\@addtoreset{equation}{section}
\makeatother

\title[The linear positive torsor presentation and arithmetic purity]
{The linear positive torsor presentation and arithmetic purity of strong approximation}

\author{Yijin Wang}
\address{Yijin WANG\newline  School of Mathematics, Shandong University,\newline  Jinan, Shandong 
Province 250100, China}
\email{202420959@mail.sdu.edu.cn}
\date{\today}
\subjclass[2020]{14G12, 14M25, 11G35}
\keywords{toric variety, strong approximation, Brauer--Manin obstruction, arithmetic purity}

\theoremstyle{definition}
\newtheorem{defi}{Definition}[section]

\newtheorem{conj}{Conjecture}

\theoremstyle{plain}
\newtheorem{thm}[defi]{Theorem}
\newtheorem{prp}[defi]{Proposition}
\newtheorem{lem}[defi]{Lemma}
\newtheorem{coro}[defi]{Corollary}

\newcommand{\Pic}{\operatorname{Pic}}

\newcommand{\codim}{\operatorname{codim}}

\newcommand{\Res}{\operatorname{Res}}

\newcommand{\A}{\mathbb A}
\newcommand{\Gm}{\mathbb G_m}
\newcommand{\olK}{\overline K}

\begin{document}

\begin{abstract}
In this article, we define a notion for varieties called linear positive torsor presentation.
For varieties with a linear positive torsor presentation over
number fields, we establish the arithmetic purity of strong approximation
with the Brauer--Manin obstruction off the empty set. As a consequence, we recover arithmetic purity for smooth toric variety with no non-constant global section. Moreover, we classify the chain connected class for rational points on such varieties over any field of characteristic zero.

\end{abstract}

\maketitle
\hypersetup{linkcolor=black,pdftitle={The linear positive torsor presentation and arithmetic purity of strong approximation},pdfauthor={Yijin Wang}}
\tableofcontents
\section{Introduction}
\subsection{Strong approximation and arithmetic purity}
Let $K$ be a number field. We denote by $\Omega_{K}$ the set of places of $K$ and $\infty_{K}$ the Archimedean places of $K$, and $\mathbb{A}^{S}_{K}$ the $S$-adele ring of $K$ for a given finite set of places $S\subset \Omega_{K}$ (possibly empty).
Let $X$ be a geometrically integral variety over $K$, we denote by $pr^{S}:X(\mathbb{A}_{K})\to X(\mathbb{A}^{S}_{K})$ the canonical projection of adelic points.

It is well known that the diagnal embedding $K\to \mathbb{A}_{K}$ has discrete image. However, the strong approximation theorem states that the embedding
\[
K\to \mathbb{A}^{S}_{K}
\]
has dense image for any $S\neq \varnothing$. This suggests the following definition:
\begin{defi} Let $X$ be a geometrically integral variety $X$ over $K$. We say that $X$ satiesfies \textbf{strong approximation off $S$} (SA) if
\[
\overline{X(K)}=pr^{S}(X(\mathbb{A}_{K})).
\]

\end{defi}
For example, the affine space $X=\mathbb{A}^{n}_{K}$ satisfies (SA) off any non-empty $S$.

The (SA) property is not a birational invariance. Wittenberg Has proposed the following conjecture in (\cite{wittenberg2016rational}, Question 2.11):
\begin{conj}
Assume $X$ satisfies strong approximation off a finite set of primes $S$ of $K$. Let $U$ be an open subset of $X$ with
$\mathrm{codim}(X-U,X)\geq2$. Does $U$ also satisfies strong approximation off $S$?
\end{conj}

We say that $X$ satiesfies \textbf{arithmetic purity of strong approximation off $S$} (APSA) once the conjecture is verifies for $X$.

The conjecture has been studied in many cases. The $X=\mathbb{A}_{K}^{n}$ case was independently verified by Cao-Xu in (\cite[Proposition 3.4]{CX18}) and Wei in (\cite{wei2021strong}, Lemma 2.1). The semi-simple simply connected quasi-split linear
algebraic groups is verified by Cao, Liang and Xu in (\cite{CLX19}).  Cao and Huang further extend it to the isotropic case in (\cite{CH20}).  
\subsection{Arithmetic purity with Brauer-Manin obstruction}
Let $X$ be a smooth geometrically integral variety over $K$. 
We denote by
\[
\mathrm{Br}(X):=H^2_{\acute{e}t}(X,\mathbb{G}_{\mathrm{m}})
\]
the cohomological Brauer group of $X$. We also define the \textbf{algebraic Brauer group} of $X$ as:
\[
\mathrm{Br}_{1}(X):=\mathrm{Ker}(\mathrm{Br}(X)\to \mathrm{Br}(\overline{X}))
\]

Given a local point $P_{v}\in X(K_{v})$, we have the evaluation map $\mathrm{Br}(X)\to \mathrm{Br}(K_{v})$ defined by the pull-back of $P_{v}:\mathrm{Spec}(K_{v})\to X$ on the cohomology groups. 
We denote the pull-back of $b\in \mathrm{Br}(X)$ by $b(P_v)$.

We then have the Brauer-Manin pairing:
$$
\langle-,-\rangle: X(\mathbb{A}_{K})\times\mathrm{Br}(X)\to \mathbb{Q}/\mathbb{Z},\  \left( (P_{v}),b\right)\mapsto \sum_{v\in \Omega_{K}}\mathrm{inv}_{v}b(P_{v}).
$$
where $\mathrm{inv}_{v}:\mathrm{Br}(K_{v})\to \mathbb{Q}/\mathbb{Z}$ is the local invariant map. 
The Brauer-Manin pairing is well defined (see \cite[Proposition 8.2.1]{Po23}).

For any subset $B\subset \mathrm{Br}(X)$, 
we define
\[
X(\mathbb{A}_{K})^{B}:=\{(P_{v})\in X(\mathbb{A}_{K})\ |\ \langle (P_{v}),b\rangle=0, \forall b\in B\}.
\]
The class field theory implies that $X(K)\subset X(\mathbb{A}_{K})^{B}$, and that $X(\mathbb{A}_{K})^{B}$ is closed in $X(\mathbb{A}_{K})$. For these facts and more about the Brauer-Manin pairing, see \cite[Chap. 8]{Po23} and \cite[Chap. 13]{CT21}. 

\begin{defi}We say that $X$ satisfies \textbf{strong approximation with Brauer-Manin obstruction} off $S$ if
\[
\overline{X(K)}=pr^{S}(X(\mathbb{A}_{K})^{\mathrm{Br}(X)}).
\]
\end{defi}

In this article, we study the following form of arithmetic purity. 
\begin{defi}
A smooth $K$-variety $X$ satisfies \textbf{arithmetic purity of strong approximation with the Brauer--Manin obstruction} ($\mathrm{APSABM}$) off $S$ if, for every open subvariety $U\subseteq X$ satisfying
\[
  \codim_X(X\setminus U)\geq 2,
\]
we have
\[
\overline{U(K)}=pr^{S}(U(\mathbb{A}_{K})^{\mathrm{Br}(U)}).
\]
\end{defi}
When the group $\mathrm{Br}$ is replaced by $\mathrm{Br_{1}}$, we say that $X$ satiesfies ($\mathrm{APSABM}_{1}$) in short.

For a smooth toric variety $X$ with $\overline{K}[X]^{\star}=\overline{K}^{\star}$, wei has showed that $X$ satiesfies $(\mathrm{APSABM}_{1})$ off \textbf{non-empty} $S$ in \cite{wei21}. For a proper smooth toric variety $X$, chen has verified that $X$ satisfies $(\mathrm{APSABM}_{1})$ off $\varnothing$ in (\cite{Chen24}). Recently, wei, xu and zhu generalized the case of proper toric varieties to the case of toric varieties with trivial global sections in (\cite{WXZ26}).

In this article, we prove the $(\mathrm{APSABM})$ property for a special class of algebraic varieties, which includes smooth toric varieties with trivial global sections.
\begin{defi}Let $X$ be a smooth geometrically integral variety over $K$. A \textbf{positive linear torsor presentation} of $X$ consisists of datum $(G,V,Y,q,\lambda)$, where

1. $G$ denotes a connected linear algebraic group or a group of multiplicative type;

2. $V$ is an affine space equipped with a linear $G$-action $\rho:G\to \mathrm{GL}(V)$; 

3. $Y\subset V$ a $G$-invariant open subscheme with $\mathrm{codim}_{V}(V-Y)\geq 2$; 

4. $q:Y\to X$ a $G$-torsor;

5. $\lambda:\mathbb{G}_{m}\to Z(G)$ a central cocharacter such that all weights of $\lambda$ on $V$ are positive.
\end{defi}
The following is the main theorem of this article.
\begin{thm}\label{main thm}
Assume that $X$ admits a positive linear torsor presentation. Then $X$ satisfies $(\mathrm{APSABM})$ off $\varnothing$. Moreover, if the structure group $G$ is of multiplicative type, then $X$ satisfies $(\mathrm{APSABM}_{1})$ off $\varnothing$.
\end{thm}

\begin{coro}[see also \cite{WXZ26}] \label{Purity of toric} Let $X$ be a smooth toric variety with $K[X]=K$, then $X$ admits a positive linear torsor presentation. Hence $X$ satiesfies $(\mathrm{APSABM}_{1})$ off $\varnothing$.
\end{coro}
\subsection{The positive linear torsor presentation and rational curves}

Let $X$ be a smooth geometrically integral variety over a field $K$ and $x,y\in X(K)$. We denote $x\sim_{\mathbb{P}^1} y$ if there is a rational curve $f:\mathbb{P}^1_{K}\to X$ such that $f(0)=x$ and $f(\infty)=y$. Then let $\sim_{\mathbb{P}^1,\mathrm{ch}}$ be the relation generated by finite chains of such morphisms. This is exactly the $R$-equivalence when $X$ is proper. 

It is very interesting to describe $X(K)/\sim_{\mathbb{P}^1\mathrm{ch}}$ for a rationally connected variety over a non-close field. For example, when $X$ is a smooth proper toric variety over a number field $K$, we have an isomorphism of pointed sets:
\begin{equation}\label{Torsor morphism}
X(K)/\sim_{\mathbb{P}^1,\mathrm{ch}}\cong H^1(K,G).
\end{equation}
where $\hat{G}=\mathrm{Pic}(\overline{X})$. See (\cite{colliot1977r}, Theorem 2, Proposition 13(ii)). 

Let us describe the morphism in (\ref{Torsor morphism}). In general, let $G$ be a linear algebraic group and $q:Y\to X$ a $G$-torsor. Then we  have a morphism of pointed sets:
\[
\delta_{q}:X(K)\to H^1(K,G), \quad x\mapsto [q^{-1}(x)].
\]
As we will see later, the morphism descents to 
\[
X(K)/\sim_{\mathbb{P}^1,\mathrm{ch}}\to H^1(K,G).
\]
In this article, we will prove the following statement:
\begin{thm}\label{P^1-connected}
Let $X/K$ be a smooth geometrically integral variety over a field of characteristic 0. Assume that $X$ admits a positive linear torsor presentation $(G,V,Y,q,\lambda)$. For $x_,y\in X(K)$, the following are equivalent:
\begin{enumerate}
\renewcommand{\labelenumi}{(\roman{enumi})}
\item $\delta_q(x)=\delta_q(y)$ in $H^1(K,G)$;
\item both $x$ and $y$ have lifts in $Y^\xi(K)$ for one twist $q^{\xi}:Y^{\xi}\to X$;
\item $x\sim_{\mathbb{P}^1} y$;
\item $x\sim_{\mathbb{P}^1,\mathrm{ch}} y$.
\end{enumerate}

Moreover, $\delta_{q}$ is surjective and induces a bijection of sets:
\[
X(K)/\sim_{\mathbb{P}^{1},\mathrm{ch}}\cong H^1(K,G).
\]
\end{thm}
\begin{coro} Let $X$ be a smooth toric variety over a field of characteristic zero with $K[X]=K$. Let $\hat{G}=\mathrm{Pic}(\overline{X})$. Then we have a natural isomorphism of sets:
\[
X(K)/\sim_{\mathbb{P}^1,\mathrm{ch}}\cong H^1(K,G).
\]
\end{coro}
\section{A scaled version of strong approximation}
In this section, we proposed a modified version of strong approximation theorem for $U\subset \mathbb{A}^{n}_{K}$ with $\mathrm{codim}_{X-U}(X)\geq 2$.
\begin{prp}\label{prp:scaled-approximation}
Let $V$ be a finite-dimensional $K$-vector space, let
\[
  \lambda:\Gm\longrightarrow\operatorname{GL}(V)
\]
be a $K$-cocharacter whose weights are strictly positive, and let $Y=V\setminus Z$ be a $\lambda$-stable open subset with $\codim_{V}Z\geq2$.

Given non-empty open subsets
\[
W_{f}\subset Y(\mathbb{A}^{f}_{K}),\quad  W_{\infty}\subset Y(K_{\infty}).
\]
There exists $y\in Y(K)$ and $r\in \mathbb{Q}_{>0}$ such that
\[
y_{f}\in W_{f},\quad y_{\infty}\in \lambda(r)W_{\infty}.
\]
\end{prp}

\begin{proof}
First assume $K=\mathbb{Q}$. Let $n=\mathrm{dim}(V)$ and 
\[
\lambda(c)(y_{1},...,y_{n})=(c^{a_{1}}y_{1},...,c^{a_{n}}y_{n})
\]

Changing the signs of the weight-basis vectors if necessary, we may assume $W_\infty\cap\mathbb R_{>0}^n\neq \varnothing$ and fix an element $w\in W_\infty\cap\mathbb R_{>0}^n$.
By (\cite{MR3873133}, proposition 4.6), we have
\[
Y(\mathbb{Q})\cap(Y^{+}(\mathbb{R})\times W_{f})\neq \varnothing
\]
Hence there exists $b\in Y(\mathbb{Q})$ such that $b_{f}\in W_{f}$ and $b_{\infty}>0$. Now consider the following morphism:
\[
\Phi_{b}:\mathbb{A}^{n}_{\mathbb{Q}}\to V\cong\mathbb{A}^{n}_{\mathbb{Q}},\quad (t_{i})\mapsto (b_{i}t^{a_{i}}_{i}).
\]
Note that $\Phi_{b}(ct)=\lambda(c)\Phi_{b}(t)$.

Let $Y_{b}:=\Phi_{b}^{-1}(Y)$, since $\Phi_{b}$ is faithfully flat, the preimage $Y_{b}$ is again an open subset of affine space with co-dimensional two complement. Moreover, $\Phi_{b}^{-1}(W_{f})$ is non-empty since it contains $1$, and $\Phi_{b}^{-1}(W_{\infty})$ is non-empty since $(\frac{w_{i}}{b_{i}})^{1/a_{i}}$ maps to $w\in W_{\infty}$.

Let $B\subset \Phi^{-1}_{b,\infty}(W_{\infty})\cap \mathbb{R}^{n}_{+}$ be a convex open ball and let $v_{1},...,v_{n}\in B\cap \mathbb{Q}^{n}$ be linearly independent vectors. Then we define an automorphism of affine space:
\[
\theta:\mathbb{A}^{n}\to \mathbb{A}^{n},\quad z\mapsto \sum^{n}_{j=1}z_{j}v_{j}
\]
For $z\in \mathbb{R}^n_{+}$, define $r(z):=\sum^{n}_{j=1}z_{j}$. Then the convexity of $B$ gives 
\[
\frac{\theta(z)}{r(z)}\in B.
\]
Now let $U_{f}=(\Phi_{b}\circ \theta)^{-1}(W_{f})$. Applying (\cite{MR3873133}, proposition 4.6) to $U_{f}$ yields an element $z_{0}\in \theta^{-1}(Y_{b})(\mathbb{Q}),$ such that
\[
z_{0}\in \mathbb{R}^{n}_{+},\quad (\Phi_{b}\circ \theta)(z_{0})\in W_{f}.
\]
Finally, let $y=\Phi_{b}(\theta(z_{0}))\in Y(\mathbb{Q})$ and $r=r(z_{0})$. Then
\[
y_{\infty}=\lambda(r)\Phi_{b}(\frac{\theta(z_{0})}{r})\in \lambda(r)\Phi_{b}(B)\subset \lambda(r)W_{\infty}.
\]
In the general case, let $d=[K:\mathbb{Q}]$ and take
\[
  \widetilde Y=\Res_{K/\mathbb{Q}}Y\subseteq
  \widetilde V=\Res_{K/\mathbb{Q}}V\simeq\A_\mathbb{Q}^{dn}.
\]
Let $Z=V\setminus Y$. Then base change to $\overline{\mathbb{Q}}$ gives
\[
  (\widetilde V\setminus\widetilde Y)_{\overline{\mathbb{Q}}}
  =\bigcup_{\sigma:K\hookrightarrow\overline{\mathbb{Q}}}
    \left(Z^\sigma\times\prod_{\tau\ne\sigma}V^\tau\right).
\]
Each term has dimension at most $(n-2)+(d-1)n=dn-2$. There are natural identifications
\[
  \begin{aligned}
    \widetilde Y(\mathbb{Q})=Y(K),\quad
    \widetilde Y(\mathbb{R})=Y(K_\infty),\quad
    \widetilde Y(\A_\mathbb{Q}^f)&=Y(\A_K^f).
  \end{aligned}
\]
Restricting the $\lambda$-action to the
diagonal $\Gm\subseteq\Res_{K/\mathbb{Q}}\Gm$ repeats every original
weight $d$ times. Hence the general case follows from the $K=\mathbb{Q}$ case.
\end{proof}

\section{Proof of the main theorem}
Let $K$ be any field of characteristic 0 and let $V$ be a variety with a left $G$-action. For $\xi\in H^1(K,G)$, let $P$ be the corresponding right torsor over $K$. Consider the contracted product:
\[
V^{\xi}=P\times^{G} V:=P\times V/\sim
\]
Here the equivalence relation is defined by $[p,x]\sim [pg,g^{-1}x]$. The variety is called the \textbf{twist} of $V$ by the torsor class $\xi$. 

Now fix a positive linear torsor presentation $(V,Y,G,q,\lambda)$. We need some easy lemma before we get started.
\begin{lem} Let $\xi\in H^1(K,G)$ and consider the twist $V^{\xi}$. Then $V^{\xi}$ is (non-canonically) isomorphic to the affine space, and $\lambda$ descends to $V^{\xi}$ as a cocharacter with positive weights.

\end{lem}
\begin{proof} As the action $\rho:G\to \mathrm{GL}(V)$ is linear, the twist $V^{\xi}$ is a vector bundle over $\mathrm{Spec}(K)$ and is isomorphic to the affine space.

Now we define a cocharacter $\lambda^{\xi}:\mathbb{G}_{m}\to \mathrm{GL}(V^{\xi})$ as:
\[
\lambda^{\xi}(t)[p,v]:=[p,\rho(\lambda(t))v].
\]
It is well-defined since $\lambda(t)\in Z(G)$.

Let $V=\bigoplus_a V_a$ for the weight decomposition under
$\rho\circ\lambda$. Each $V_a$ is
$G$-stable Since $\lambda$ is central. Hence
\[
V^\xi=\bigoplus_a (P\times^G V_a).
\]
The induced action of $\mathbb G_m$ on $P\times^G V_a$ is multiplication by $t^a$, and
$\dim_K(P\times^G V_a)=\dim_K V_a$. Thus $\lambda^\xi$ has the same weights as
$\rho\circ\lambda$ with the same multiplicities. In particular,
all weights of $\lambda^{\xi}$ are positive.
\end{proof}
Let $U\subseteq X$ be an open subvariety  with
$\operatorname{codim}_X(X\setminus U)\geq2$, and let
$Y_U=q^{-1}(U)$.
\begin{lem}\label{codim 2}
The open immersion $Y_U\subseteq V$ satisfies
\[
  \codim_V(V\setminus Y_U)\geq2.
\]
\end{lem}

\begin{proof}
Since $q$ is a torsor under the smooth group $G$, it is smooth and faithfully flat.  Therefore
\[
  \codim_Y q^{-1}(X\setminus U)
  =
  \codim_X(X\setminus U)
  \geq2.
\]
Let $Z$ be the closure of $q^{-1}(X\setminus U)$ in $V$.  If $Z$ contains a codimension-one irreducible component $E$, then $E$ could not be contained in $V\setminus Y$.  Hence $E\cap Y$ would be a codimension-one subset of $Y$ contained in $q^{-1}(X\setminus U)$, a contradiction, and the result follows.
\end{proof}
\begin{proof}[Proof of Theorem \ref{main thm}]
Let $x\in U(\mathbb A_K)^{\operatorname{Br}(U)}$. When $G$ is of multiplicative type, we choose $x\in U(\mathbb{A}_{K})^{\mathrm{Br}_{1}(U)}$. Let
$\mathcal N$ be an open neighbourhood of $x$ in
$U(\mathbb A_K)$.

By (\cite{CDX19}, Theorem 3.5) , there exist $\xi\in H^1(K,G)$ and 
  $\eta\in Y_U^\xi(\mathbb A_K)$,
such that $p_U^\xi(\eta)=x$. Let $f=p^{\xi}_{U}$.

The central cocharacter descends to $V^\xi$ with the same
positive weights. We continue to denote it by $\lambda$;
it preserves $Y^{\xi}_{U}$ and satisfies
\[
  f(\lambda(t)y)=f(y).
\]
There exists nonempty open subsets
\[
  W_f\subseteq Y^{\xi}_{U}(\mathbb A_K^f),
  \qquad
  W_\infty\subseteq Y^{\xi}_{U}(K_\infty)
\]
such that
\[
  \eta\in W_f\times W_\infty
  \subseteq f^{-1}(\mathcal N).
\]
By Proposition \ref{prp:scaled-approximation}, there exist
$y\in Y^{\xi}_{U}(K)$ and $r\in\mathbb Q_{>0}$ satisfying
\[
  y_f\in W_f,\qquad
  y_\infty\in\lambda(r)W_\infty.
\]
Consequently,
\[
  \eta':=(y_f,\lambda(r)^{-1}y_\infty)
  \in W_f\times W_\infty.
\]
Since $f$ is invariant under $\lambda$, the element
$u:=f(y)=f(\eta')\in\mathcal N$. And we conclude the theorem as $x$ and $\mathcal N$ were arbitrary.

\end{proof}
\section{Arithmetic purity for toric varieties}

Let $X$ be a smooth toric variety with $K[X]=K$. In this section, we  construct a positive linear torsor presentation for $X$.

$T\subseteq X$ be the dense torus. We denote by
\[
  M=X^*(T_{\olK}),
  \qquad
  N=X_*(T_{\olK}),
\]
the Character and cocharacter group of $T$. Let $\Sigma$ be the fan of $X_{\overline{K}}$ in $N_{\mathbb R}$. For basic facts about fans and toric varieties over non-close fields, see for example \cite{huruguen2011toric}.

\subsection{The universal torsor of toric variety}

Let $\Sigma(1)$ denote the set of rays of $\Sigma$.  For $\rho\in\Sigma(1)$, let $v_\rho\in N$ be its primitive generator and let $D_\rho$ be the corresponding torus-invariant prime divisor.  Let
\[
  \mathcal D=\bigoplus_{\rho\in\Sigma(1)}\mathbb ZD_\rho.
\]
We have the following exact sequence:
\[
  0\longrightarrow M
  \stackrel{\delta}{\longrightarrow}
  \mathcal D
  \longrightarrow P
  \longrightarrow 0,
  \qquad
  P=\Pic(X_{\olK}),
\]
where
\[
  \delta(m)=\sum_{\rho\in\Sigma(1)}
  \langle m,v_\rho\rangle D_\rho.
\]

Let $R$ be the quasi-trivial torus with character module $\mathcal D$, and let $G$ be the group of multiplicative type with character module $P$.  Dualizing the sequence gives an exact sequence of algebraic groups 
\[
  1\longrightarrow G
  \longrightarrow R
  \longrightarrow T
  \longrightarrow 1.
\]
Note that the module $P$ may have torsion

The finite Galois set $\Sigma(1)$ defines a $K$-vector space $V$ such that
\[
  V_{\olK}\simeq\A_{\olK}^{\Sigma(1)}.
\]
Write $z_\rho$ for its geometric coordinates.  The torus $R_{\olK}$ acts by coordinatewise multiplication.  Define
\[
  Y_{\olK}
  =
  \bigcup_{\sigma\in\Sigma}
  \left\{
  z_\rho\neq 0\text{ for every }\rho\notin\sigma(1)
  \right\}.
\]
This open subset is Galois stable and therefore descends to a $K$-open subset $Y\subseteq V$.

See (\cite{cox1995homogeneous}) and (\cite{santens2026manin}, construction 2.9, proposition 2.10) for the next proposition.
\begin{prp}\label{cox torsor}
There is a natural $G$-torsor
\[
  p:Y\longrightarrow X.
\]
Moreover,
\[
  \codim_V(V\setminus Y)\geq 2.
  \tag{5.5}\label{eq:cox-codim}
\]
\end{prp}

\subsection{A positive cocharacter}

\begin{lem} \label{positive weight}
There are positive integers $a_\rho>0$, constant on Galois orbits in $\Sigma(1)$, such that
\[
  \sum_{\rho\in\Sigma(1)}a_\rho v_\rho=0.
  \tag{5.7}\label{positive-relation}
\]
Moreover, there is a $K$-cocharacter
\[
  \lambda:\Gm\longrightarrow G
\]
whose action on $V_{\olK}$ is
\[
  \lambda(t)\cdot z_\rho=t^{a_\rho}z_\rho.
\]
In particular, all weights of $\lambda$ on $V$ are strictly positive.
\end{lem}

\begin{proof}
. Let
$C=\sum_{\rho\in\Sigma(1)}\mathbb R_{\geq0}v_\rho
\subseteq N_{\mathbb R}$,
and let $C^\vee\subseteq M_{\mathbb R}$ be its dual cone.
Inside the group algebra $\overline{K}[M]$, restriction
to the torus identifies
\[
\overline K[X_{\overline K}]
=
\bigcap_{\sigma\in\Sigma}
\overline K[\sigma^\vee\cap M]
=
\overline K[C^\vee\cap M]
\subseteq \overline K[M].
\]
It follows that $C^\vee\cap M=\{0\}$,
and we conclude that $C^{\vee}=0$ since $C^\vee$ is a rational polyhedral cone. Thus $C=(C^\vee)^\vee=N_{\mathbb R}$ by the duality of closed
convex cones.

Consequently,  we have
$-v_\rho\in C$ for every $\rho\in\Sigma(1)$. Choose $b_{\rho\tau}\in\mathbb Q_{\geq0}$ such that
$v_\rho+\sum_{\tau\in\Sigma(1)}b_{\rho\tau}v_\tau=0$.
Summing these relations over $\rho$, we obtain
\[
\sum_{\tau\in\Sigma(1)}
\left(1+\sum_{\rho\in\Sigma(1)}b_{\rho\tau}\right)v_\tau=0.
\]
Every coefficient is a strictly positive rational number. Clearing denominators gives positive integers $n_\tau$ such that
$\sum_\tau n_\tau v_\tau=0$.

Let $L/K$ be the splitting field of $T$ and let $\Gamma=\mathrm{Gal}(L/K)$. For each $\rho$, let
$a_\rho=\sum_{\gamma\in\Gamma}n_{\gamma^{-1}\rho}$, then we have
\[
\sum_{\rho\in\Sigma(1)}a_\rho v_\rho
=
\sum_{\gamma\in\Gamma}
\sum_{\tau\in\Sigma(1)}n_\tau v_{\gamma\tau}
=
\sum_{\gamma\in\Gamma}
\gamma\left(\sum_{\tau\in\Sigma(1)}n_\tau v_\tau\right)
=0.
\]

Define $\ell:D\to\mathbb Z$ by $\ell(D_\rho)=a_\rho$, then $\ell$ is
Galois-equivariant. For every $m\in M$, we have
\[
\ell(\delta(m))
=
\sum_{\rho\in\Sigma(1)}
a_\rho\langle m,v_\rho\rangle
=
\left\langle m,\sum_{\rho\in\Sigma(1)}a_\rho v_\rho\right\rangle
=0.
\]
Hence $\ell$ factors through a Galois-equivariant homomorphism
$\overline\ell:P\to\mathbb Z$.
Taking the dual of $\overline{\ell}$ yields a
$K$-cocharacter $\lambda:\mathbb G_m\to G$.

The composite $\mathbb G_m\xrightarrow{\lambda}G\hookrightarrow R$
is the dual of $\ell:D\to\mathbb Z$.
Consequently, after base change to $\overline K$, the map sends
$t$ to $(t^{a_\rho})_{\rho\in\Sigma(1)}$, and its induced action on $V_{\overline{K}}$ is
\[
(z_\rho)_{\rho\in\Sigma(1)}
\longmapsto
(t^{a_\rho}z_\rho)_{\rho\in\Sigma(1)}
\]
which proves our claim.
\end{proof}
\begin{proof}[Proof of corollary \ref{Purity of toric}] Let $q:Y\to X$ be the torsor defined in propostion \ref{cox torsor}, and $\lambda:\mathbb{G}_{m}\to G$ be the cocharacter defined in lemma \ref{positive weight}. Then the datum $(G,V,Y,q,\lambda)$
 is a postive linear torsor presentation for $X$. Since $G$ is of multiplicative type, we conclude the result by theorem \ref{main thm}.
\end{proof}
\section{Rational curves and linear positive torsor presentation}
Let $X$ be a geometrically integral variety with a positive linear torsor presentation $(G,V,Y,q,\lambda)$ over $K$ with $\mathrm{char}(K)=0$, we explain when is it possible to connect two rational points $x,y\in X(K)$ by $\mathbb{P}^1_{K}$.
\begin{lem}\label{avoid}
Let $U\subseteq W\simeq\A_K^n$ be a nonempty open subset such that $\codim_W(W\setminus U)\ge2$. Assume moreover $U$ is stable under
\[
 \lambda(r)(z_1,\ldots,z_n)=(r^{a_1}z_1,\ldots,r^{a_n}z_n),
 \qquad a_i\in\mathbb{Z}_{>0}.
\]
 Let $p:U\to X$ be a $K$-morphism invariant under the action of $\lambda$. Then for any $u,v\in U(K)$, there is a $K$-morphism $f:\mathbb{P}^1_K\to X$ with
\[
 f(p_0)=p(u),\qquad f(p_\infty)=p(v).
\]
\end{lem}
\begin{proof}
Let $Z=W\setminus U$ with its reduced scheme structure, and
let $u=(u_i)$ and $v=(v_i)$ be their coordinates. Let $C=\mathbb A^n_K$ with
coordinates $c_1,\ldots,c_n$ and let $[S:T]$ be the homogeneous coordinates of $\mathbb P^1_K$
with $0=[1:0]$ and $\infty=[0:1]$. Define
\[
F:C\times_K(\mathbb A^2_K\setminus\{0\})\longrightarrow W,
\qquad
F(c,S,T)=
\bigl(u_iS^{2a_i}+c_i(ST)^{a_i}+v_iT^{2a_i}\bigr)_{i=1}^n.
\]
Then $F(c,1,0)=u$, $F(c,0,1)=v$, and $F(c,rS,rT)=\lambda(r^2)F(c,S,T)$.

The ideal of $Z$ is homogeneous for
the grading $\deg z_i=a_i$ since $Z$ is $\lambda$-stable. Thus we may choose weighted homogeneous
generators $h_1,\ldots,h_b$ of weighted degrees
$d_1,\ldots,d_b$. Each $h_j(F(c,S,T))$ is homogeneous of
degree $2d_j$ in $S,T$. Thus these equations define a
closed subscheme
\[
\mathcal M=
\bigl\{(c,[S:T])\in C\times_K\mathbb P^1_K:
F(c,S,T))\in Z\}.
\]
Note that the fibres of $\mathcal{M}\to \mathbb{P}^1$ over $0$ and $\infty$ are empty, since
$u,v\notin Z$.

We identify $D_+(ST)\subseteq\mathbb P^1_K$ with $\mathbb G_m$
via $\tau=T/S$. Consider the morphism
\[
\begin{aligned}
\Psi:C\times_K\mathbb G_m
&\longrightarrow W\times_K\mathbb G_m,\\
(c,\tau)
&\longmapsto
\left(
\bigl(u_i+c_i\tau^{a_i}+v_i\tau^{2a_i}\bigr)_{i=1}^n,
\tau
\right).
\end{aligned}
\]
It is an isomorphism, with inverse
\[
(z,\tau)\longmapsto
\left(
\bigl(\tau^{-a_i}(z_i-u_i-v_i\tau^{2a_i})\bigr)_{i=1}^n,
\tau
\right).
\]

Under the identification $D_+(ST)\simeq\mathbb G_m$, we have
\[
\mathcal M\cap(C\times_K D_+(ST))
=\Psi^{-1}(Z\times_K\mathbb G_m).
\]
In particular, we have $\dim\mathcal M=\dim Z+1\leq n-1$.

Let $pr_{2}:C\times_K\mathbb P^1_K\to C$
and $B=\operatorname{pr}_2(\mathcal M)$. 
Then $B\neq C$, and we may choose a rational point
$c_0\in(C\setminus B)(K)$.
Then $F(c_0,S,T)\notin Z$ for every geometric point
$(S,T)\neq(0,0)$, thus $F(c_0,-,-)$ factors through $U$.

On the charts $S\neq0$ and $T\neq0$ of $\mathbb{P}_{K}^1$, we write the coordinates
$\tau=T/S$ and $\sigma=S/T$ respectively, and define
\[
\begin{aligned}
f_S:\mathbb A^1_{K,\tau}&\longrightarrow X,
&\tau&\longmapsto p(F(c_0,1,\tau)),\\
f_T:\mathbb A^1_{K,\sigma}&\longrightarrow X,
&\sigma&\longmapsto p(F(c_0,\sigma,1)).
\end{aligned}
\]
On the overlap, we have $\sigma=\tau^{-1}$ and
$F(c_0,\sigma,1)
=\lambda(\sigma^2)F(c_0,1,\sigma^{-1})$.
The $\lambda$-invariance of $p$ gives
$f_T(\sigma)=f_S(\sigma^{-1})$.
Thus these morphisms glue to a $K$-morphism
$f:\mathbb P^1_K\to X$ satisfying
$f(0)=p(u)$ and $f(\infty)=p(v)$.
\end{proof}

\begin{proof}[Proof of Theorem \ref{P^1-connected}]
The equivalence of (i) and (ii) follows from the fact that:
\[
x\in q^{\xi}(Y^{\xi}(K))\Leftrightarrow [Y_{x}]=\xi\in H^1(K,G).
\]

(ii) implies (iii): This is a direct consequence of (lemma \ref{avoid}).

(iii) implies (iv): Trivial.

(iv) implies (i): If $x,y$ is connected by one chain, let $f:\mathbb P^1_K\to X$ with $f(0)=x$ and $f(\infty)=y$.
Denote $P=Y_x$. The Scheme
\[
  \mathcal T
  :=\underline{\operatorname{Isom}}_G
       (P\times_K\mathbb P^1_K,f^*Y)
\]
is a torsor under the smooth $K$-group $\operatorname{Aut}_G(P)$ with
 trivial fibre at $0$.
By \cite[Theorem 7.2.1]{bouthier2025generically}, its fibre at $\infty$ is
also trivial. Thus $Y_y\simeq P=Y_x$, and hence
\[
  \delta_q(x)=\delta_q(y).
\]
In the general case, apply the argument repeatedly. 

Finally, let $\xi\in H^1(K,G)$ be an arbitrary element. By Lemma 3.1,
$V^\xi$ is isomorphic to $\mathbb A_K^{\dim_K V}$.
Since $K$ is infinite, every nonempty
open subset of affine space has a $K$-point. Therefore we may
choose $y\in Y^\xi(K)$ and let $x=q^\xi(y)\in X(K)$. Then
we have $\delta_q(x)=\xi$, thus $\delta_q$ is surjective.
Together with the equivalence of (i) and (iv), we conclude that $X(K)/\sim_{\mathbb{P}^1,\mathrm{ch}} \cong  H^1(K,G)$.
\end{proof}

\textbf{Acknowledgments.} The authors thank Yang cao for helpful discussions.
The author is partially supported by National Key R\&D Program of China Grant No. 2025YFA1017300 
(subprogram No. 2025YFA1017303).

\textbf{The usage of LLM model.} The middle part of the proof of Lemma 5.1 was completed by ChatGPT 6.0 Astra. All other parts were completed by the author.

\bibliographystyle{alpha}
\bibliography{ref}

@article {CDX19,
    AUTHOR = {Cao, Yang and Demarche, Cyril and Xu, Fei},
     TITLE = {Comparing descent obstruction and {B}rauer-{M}anin obstruction
              for open varieties},
   JOURNAL = {Trans. Amer. Math. Soc.},
  FJOURNAL = {Transactions of the American Mathematical Society},
    VOLUME = {371},
      YEAR = {2019},
    NUMBER = {12},
     PAGES = {8625--8650},
      ISSN = {0002-9947,1088-6850},
   MRCLASS = {14G05 (14F22)},
  MRNUMBER = {3955558},
MRREVIEWER = {Brendan\ Creutz},
       DOI = {10.1090/tran/7567},
       URL = {https://doi.org/10.1090/tran/7567},
}

@book {CT21,
    AUTHOR = {Colliot-Th\'el\`ene, Jean-Louis and Skorobogatov, Alexei N.},
     TITLE = {The {B}rauer-{G}rothendieck group},
    SERIES = {Ergebnisse der Mathematik und ihrer Grenzgebiete. 3. Folge. A
              Series of Modern Surveys in Mathematics [Results in
              Mathematics and Related Areas. 3rd Series. A Series of Modern
              Surveys in Mathematics]},
    VOLUME = {71},
 PUBLISHER = {Springer, Cham},
      YEAR = {2021},
     PAGES = {xv+453},
      ISBN = {978-3-030-74247-8; 978-3-030-74248-5},
   MRCLASS = {14F22 (14E08 14G05 14G12 14K05)},
  MRNUMBER = {4304038},
MRREVIEWER = {Thomas\ Benedict\ Williams},
       DOI = {10.1007/978-3-030-74248-5},
       URL = {https://doi.org/10.1007/978-3-030-74248-5},
}

@book {Po23,
    AUTHOR = {Poonen, Bjorn},
     TITLE = {Rational points on varieties},
    SERIES = {Graduate Studies in Mathematics},
    VOLUME = {186},
 PUBLISHER = {American Mathematical Society, Providence, RI},
      YEAR = {2017},
     PAGES = {xv+337},
      ISBN = {978-1-4704-3773-2},
   MRCLASS = {14G05 (11G35)},
  MRNUMBER = {3729254},
MRREVIEWER = {Daniel\ Loughran},
       DOI = {10.1090/gsm/186},
       URL = {https://doi.org/10.1090/gsm/186},
}

@article{wittenberg2016rational,
  title={Rational points and zero-cycles on rationally connected varieties over number fields},
  author={Wittenberg, Olivier},
  journal={arXiv preprint arXiv:1604.08543},
  year={2016}
}

@inproceedings{CX18,
  title={Strong approximation with Brauer--Manin obstruction for toric varieties},
  author={Cao, Yang and Xu, Fei},
  booktitle={Annales de l'Institut Fourier},
  volume={68},
  number={5},
  pages={1879--1908},
  year={2018}
}

@article{wei2021strong,
  title={Strong approximation for a toric variety},
  author={Wei, Da Sheng},
  journal={Acta Mathematica Sinica, English Series},
  volume={37},
  number={1},
  pages={95--103},
  year={2021},
  publisher={Springer}
}

@article{CLX19,
  title={Arithmetic purity of strong approximation for homogeneous spaces},
  author={Cao, Yang and Liang, Yongqi and Xu, Fei},
  journal={Journal de Math{\'e}matiques Pures et Appliqu{\'e}es},
  volume={132},
  pages={334--368},
  year={2019},
  publisher={Elsevier}
}

@article{CH20,
  title={Arithmetic purity of strong approximation for semi-simple simply connected groups},
  author={Cao, Yang and Huang, Zhizhong},
  journal={Compositio Mathematica},
  volume={156},
  number={12},
  pages={2628--2649},
  year={2020},
  publisher={London Mathematical Society}
}

@article{Chen24,
  title={Arithmetic purity of strong approximation for complete toric varieties},
  author={Chen, Sheng},
  journal={Mathematische Zeitschrift},
  volume={306},
  number={2},
  pages={23},
  year={2024},
  publisher={Springer}
}

@article{WXZ26,
  title={Arithmetic purity of strong approximation for toric varieties with constant global sections},
  author={Wei, Dasheng and Xu, Fei and Zhu, Yi},
  journal={arXiv preprint arXiv:2608.28204},
  year={2026}
}

@article{santens2026manin,
  title={Manin's conjecture for integral points on toric varieties},
  author={Santens, Tim},
  journal={Proceedings of the London Mathematical Society},
  volume={133},
  number={2},
  pages={e70203},
  year={2026},
  publisher={Wiley Online Library}
}

@article {MR3873133,
    AUTHOR = {Cao, Yang and Xu, Fei},
     TITLE = {Strong approximation with {B}rauer-{M}anin obstruction for
              groupic varieties},
   JOURNAL = {Proc. Lond. Math. Soc. (3)},
  FJOURNAL = {Proceedings of the London Mathematical Society. Third Series},
    VOLUME = {117},
      YEAR = {2018},
    NUMBER = {4},
     PAGES = {727--750},
      ISSN = {0024-6115,1460-244X},
   MRCLASS = {14G05 (11G35 14F22 14G25 20G15)},
  MRNUMBER = {3873133},
MRREVIEWER = {Sho\ Tanimoto},
       DOI = {10.1112/plms.12141},
       URL = {https://doi.org/10.1112/plms.12141},
}

@article{bouthier2025generically,
  title={Generically trivial torsors under constant groups},
  author={Bouthier, Alexis and Cesnavicius, Kestutis and Scavia, Federico},
  journal={arXiv preprint arXiv:2505.00505},
  year={2025}
}

@article {wei21,
    AUTHOR = {Wei, Da Sheng},
     TITLE = {Strong approximation for a toric variety},
   JOURNAL = {Acta Math. Sin. (Engl. Ser.)},
  FJOURNAL = {Acta Mathematica Sinica (English Series)},
    VOLUME = {37},
      YEAR = {2021},
    NUMBER = {1},
     PAGES = {95--103},
      ISSN = {1439-8516,1439-7617},
   MRCLASS = {11G35 (14G05 14M25)},
  MRNUMBER = {4204537},
MRREVIEWER = {Zhizhong\ Huang},
       DOI = {10.1007/s10114-021-8193-7},
       URL = {https://doi.org/10.1007/s10114-021-8193-7},
}

@inproceedings{colliot1977r,
  title={La $ R $-{\'e}quivalence sur les tores},
  author={Colliot-Th{\'e}l{\`e}ne, Jean-Louis and Sansuc, Jean-Jacques},
  booktitle={Annales scientifiques de l'{\'E}cole Normale Sup{\'e}rieure},
  volume={10},
  number={2},
  pages={175--229},
  year={1977}
}

@article{cox1995homogeneous,
  title={The homogeneous coordinate ring of a toric variety},
  author={Cox, David A},
  journal={Journal of Algebraic Geometry},
  volume={4},
  number={1},
  pages={17--50},
  year={1995},
  publisher={Providence, RI: University Press, c1992-}
}

@article{huruguen2011toric,
  title={Toric varieties and spherical embeddings over an arbitrary field},
  author={Huruguen, Mathieu},
  journal={Journal of Algebra},
  volume={342},
  number={1},
  pages={212--234},
  year={2011},
  publisher={Elsevier}
}
\end{document}